\documentclass[12pt]{elsarticle}
\usepackage[cp1251]{inputenc}
\usepackage[english,russian]{babel}
\usepackage{babel}
\usepackage{amsmath}
\usepackage{amsfonts}
\usepackage{lineno}

\usepackage{algorithm}
\usepackage[noend]{algorithmic}

\newtheorem{theorem}{Theorem}
\newtheorem{lemma}{Lemma}
\newproof{proof}{Proof}
\newdefinition{note}{Note}

\newcommand{\cone}{\mathop{\rm cone}}
\newcommand{\conv}{\mathop{\rm conv}}
\newcommand{\ZZ}{\mathbb{Z}}
\newcommand{\QQ}{\mathbb{Q}}
\newcommand{\RR}{\mathbb{R}}

\DeclareMathOperator{\poly}{poly}
\DeclareMathOperator{\defect}{def}
\DeclareMathOperator{\size}{size}
\DeclareMathOperator{\diag}{diag}
\DeclareMathOperator{\linh}{span}

\DeclareMathOperator{\MainAlgName}{Min-DConic}
\DeclareMathOperator{\PrepAlgName}{Preprocessing}

\usepackage{color}
\usepackage{xcolor}

\newcommand{\NewAdd}[1]{{#1}}

\journal{Discrete Applied Mathematics}

\begin{document}

\begin{frontmatter}



\title{\NewAdd{A polynomial algorithm for minimizing discrete \NewAdd{convic} functions in~fixed~dimension}\tnoteref{title_ref}}
\tnotetext[title_ref]{The work was supported by the Russian Science Foundation Grant No. 17-11-01336.}

%


\author{S.~I.~Veselov}
\ead{Sergey.Veselov@itmm.unn.ru}
\author{D.~V.~Gribanov}
\ead{Dmitry.Gribanov@itmm.unn.ru}
\author{N.~Yu.~Zolotykh}
\ead{Nikolai.Zolotykh@itmm.unn.ru}
\author{A.~Yu.~Chirkov}
%\ead{Aleksandr.Chirkov@itmm.unn.ru}

\address{%Department of algebra, geometry and discrete mathematics,\\ 
%Institute of information technology, mathematics and mechanics, \\ 
Lobachevsky State University of Nizhni Novgorod,\\  
23 Gagarin Ave., Nizhny Novgorod, 603950, Russia}

\begin{abstract}
In \cite{JOGO}, classes of {\em conic} and {\em discrete conic} functions were introduced.
\NewAdd{In this paper we use the term {\em convic} instead {\em conic}.}
The class of \NewAdd{convic} functions properly includes the classes of convex functions, strictly quasiconvex functions and the class of quasiconvex polynomials. On the other hand, the class of \NewAdd{convic} functions is properly included in the class of quasiconvex functions.
The discrete \NewAdd{convic} function is a discrete analogue of the \NewAdd{convic} function.
In \cite{JOGO}, the lower bound $3^{n-1}\log (2 \rho-1)$ for the number of calls to the comparison oracle needed to find the minimum 
of the discrete \NewAdd{convic} function defined on integer points of some $n$-dimensional ball with radius $\rho$ was obtained.
But the problem of the existence of a polynomial (in $\log\rho$ for fixed $n$)
algorithm for minimizing such functions has remained open. 
%No polynomial (in terms of $\log\rho$ for fixed $n$) algorithm for minimizing such a function was been known.
In this paper, we answer positively the question of the existence of such an algorithm.
Namely, we propose an algorithm for minimizing discrete \NewAdd{convic} functions that uses $2^{O(n^2 \log n)} \log \rho$ calls to the comparison oracle
\NewAdd{and has $2^{O(n^2 \log n)} \poly(\log \rho)$ bit complexity}.
\end{abstract}


\end{frontmatter}

%\linenumbers

\section{Introduction} 

A well-known and intensively studied generalization of the integer linear programming problem is
the problem of integer minimization of (quasi-) convex functions subject to (quasi-) convex constraints
\cite{
	Chirkov2003,DadushPeikertVempala2011,Dadush,Heinz2005,Heinz2008,HemmeckeOnnWeismantel2011,HildebrandKoppe2013,Loera2006,OertelWagnerWeismantel2014}. 
The objective function and the constraints in the problem can be specified explicitly or using an oracle.
In \cite{DadushPeikertVempala2011,OertelWagnerWeismantel2014}, a polynomial algorithm in terms of $\log \rho $ 
(when the dimension $n$ is fixed) is proposed, if the domain of the function is contained in a ball of radius $\rho$ and the function is specified by the separation hyperplane oracle.
In \cite{OertelWagnerWeismantel2014}, it is established that a polynomial algorithm in terms of $\log \rho$ (when $n$ is fixed) can be obtained for the next three oracles: the feasibility oracle, the linear integer
optimization oracle and the separation hyperplane oracle.
A new approach in integer convex optimization, based on the concept of a centerpoint, is proposed in \cite{BasuOertel2017}.
In those papers, the lower bound $\Omega (2^n \log \rho)$ for the complexity of the algorithms of integer convex minimization using the separating hyperplane oracle is also established.

The main disadvantage of the oracles mentioned above is the complexity of their implementation.
\NewAdd{In many situations, it is more convenient to use} the comparison oracle and the $0$-order oracle.
For any two points $x$, $y$ in the domain of $f$, the comparison oracle allows 
us to determine which of the two inequalities is satisfied:
$f(x) \le f(y)$ or $f(x) > f(y)$. Given $x$, the $0$-order oracle returns $f(x)$.
Note that, as shown in \cite{JOGO}, the separating hyperplane oracle for the problems under consideration cannot be polynomially reduced to the comparison oracle.

The problems of integer minimization of convex (and close to them) functions defined by the comparison oracle or/and by the $0$-order oracle are considered in \cite{Chirkov2003,JOGO,Veselov2018,ZolotykhChirkov2012}.
In \cite{Chirkov2003}, an algorithm for integer minimization of symmetric strictly quasiconvex functions with $n = 2$ with the number of calls to the comparison oracle at most $2\log_2^2\rho + 22\log_2\rho$ is proposed.
In \cite{Veselov2018}, a similar problem with the $0$-order oracle was considered and an algorithm for integer minimization of such functions was proposed with the number of calls to the oracle at most $ 4\log_2 \rho$.
In addition, the lower bound $1.44 \log_2 \rho -2 $ for this problem was obtained in \cite{Veselov2018}.

The paper \cite{JOGO} discusses the possibility of extending the class of functions
for which the integer optimization problem \NewAdd{with} the comparison oracle 
can still be solved in polynomial in $\log \rho$ time for any fixed $n$.
In particular, in \cite{JOGO}, the classes of conic and discrete conic functions \NewAdd{were} introduced. \NewAdd{The term {\em conic} in this context led to confusion and misunderstanding, so here, instead of {\em conic}, we will use the invented word {\em convic} (derived from {\em conic} and {\em convex}).} The definitions of convic and discrete convic functions see in Section \ref{sec_def}.

The class of \NewAdd{convic} functions contains, as proper subclasses, the following important classes of functions: convex functions, strictly quasiconvex functions, and quasiconvex polynomials.
The class of discrete \NewAdd{convic} functions is a discrete analogue of the class of \NewAdd{convic} functions.

In \cite{JOGO}, an algorithm for minimizing \NewAdd{convic} functions using 
$2^{O(n \log n)} \log \rho$ calls to the comparison oracle is proposed.
But the problem of the existence of a polynomial (in $\log\rho$ for fixed $n$)
algorithm for minimizing {\em discrete} \NewAdd{convic} functions has remained open.
\NewAdd{We note that the algorithm in \cite{JOGO} is not applied to discrete \NewAdd{convic}
functions because it may appeal to the oracle at fractional points. Now we deal with discrete \NewAdd{convic} functions and have to call to the oracle at only integer points.}
On the other hand, in \cite{JOGO}, generalizing results of \cite{ZolotykhChirkov2012}, a lower bound $3^{n-1}\log (2 \rho-1)$ for this problem was obtained.

In this paper, we answer positively to the question of the existence of a polynomial algorithm for minimizing {\em discrete} \NewAdd{convic} functions using the comparison oracle.
Namely, we propose such a minimization algorithm with the number of calls to the oracle at most \NewAdd{$2^{O(n^2 \log n)} \log \rho$ and bit complexity 
$2^{O(n^2 \log n)} \poly(\log \rho)$}.

As in \cite{DadushPeikertVempala2011,HildebrandKoppe2013,OertelWagnerWeismantel2014} 
when constructing our algorithm, we use 
the modified Lenstra method ($\beta$-rounding algorithm) for solving the feasibility problem
described in general form in \cite{Grotschel, Schrijver}. 
It uses two main ideas: the ellipsoid method \cite{NemirovskyYudin1983}
and the lattice basis reduction \cite{LLL1982}. 
Since, in the problems under consideration, there is no function defined analytically, these problems cannot be reduced to the problem of satisfiability. Nevertheless, due to the properties of discrete \NewAdd{convic} functions, the $\beta$-rounding procedure can be applied here.
Also note that in the papers cited above, the separation oracle is used, and the opportunity to ask questions to the oracle at any rational points of the domain is essentially used.
The difference of our work, among others, is that it is enough for our oracle to ask queries only at integer points. 
This allows one to work with partially defined functions, as well as with functions for which the separation oracle is complex. 
We achieve this due to a more subtle use of the properties of the reduced basis and the $\beta$-rounding algorithm.

The paper is organized as follows.
Section \ref{sec_def} introduces the necessary definitions and notation.
Section \ref{sec_beta_rounding} contains a general description of the $\beta$-rounding algorithm.
In Section \ref{sec_hyperplain}, the equations of the new ellipsoid and the cut-off hyperplane at the current iteration of the $\beta$-rounding algorithm are derived.
Section \ref{sec_algorithm} contains a description of the main algorithm for minimizing discrete \NewAdd{convic} functions and a bound for the algorithm complexity.

\section{Main definitions and notations\label{sec_def}}

Let $X$ be a subset of $\RR^n$.
Denote by $\linh X$, $\cone X$ and $\conv X$ the linear, conic and convex hulls of $X$, respectively.

A function $f:~ \RR^n \to \RR$ is called {\em quasiconvex}, if for any $c$ the set
$M_c = \{x\in\RR^n: f(x)\le c\}$ is convex.
Let $D$ be a discrete set in $\RR^n$.
\NewAdd{A function $f:~ D \to \RR$ is called {\em discrete quasiconvex}, if for any $c$ the set	$M_c = \{x\in D : f(x)\le c\}$ is {\em discrete convex}, i.e. $\conv(M_c)\cap D  = M_c$.}

A function $f:~ \RR^n \to \RR$ is called {\em \NewAdd{convic}}, if for any $y\in\RR^n$ and for any $z\in\RR^n$, 
belonging to the set 
$y + \cone\{y - x:~ f(x) \le f(y),~ x \in \RR^n \}$, it holds that $f(y) \le f(z)$ \cite{JOGO}.

A function $f:~   D \to \RR$ is called {\em discrete \NewAdd{convic}}, if for any $y\in D$ and any $z\in D$, belonging to the set
\NewAdd{$y + \cone\{y - x:~ f(x) \le f(y), ~x\in D \}$}, it holds that $f(y) \le f(z)$ \cite{JOGO};
see Fig.\,\ref{discrete_convic_fig}.

\begin{figure}\label{discrete_convic_fig}
	\centering
	\includegraphics[clip,scale=0.38]{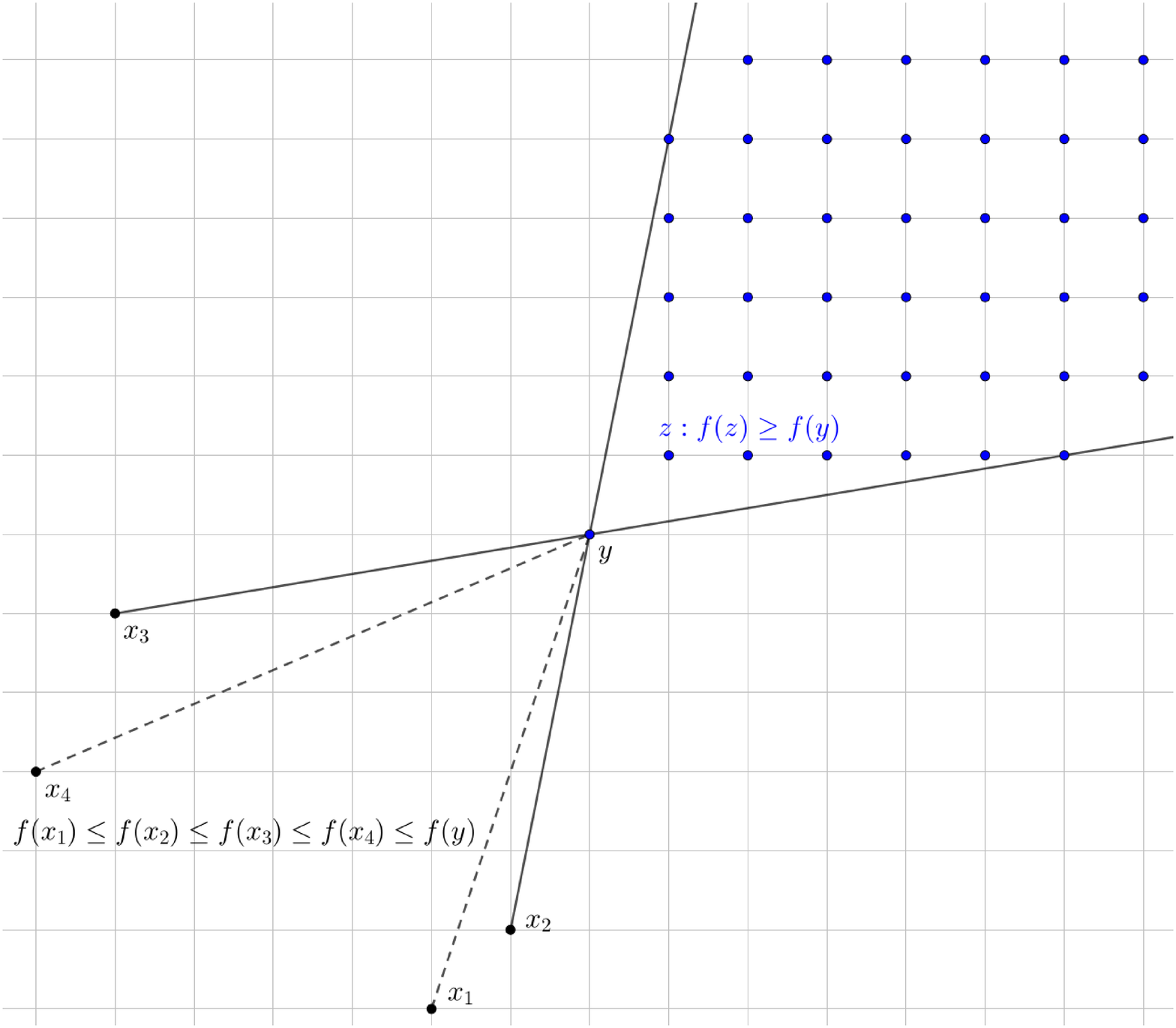}
	\caption{An illustration of the discrete \NewAdd{convic} function definition. Here, we have $f : \ZZ^2 \to \RR$ and $\max_{i} f(x_i) \leq f(y)$.}
\end{figure}

In this paper, we consider the problem of minimization of a discrete \NewAdd{convic} function defined on a set $D = \ZZ^n \cap B_{\rho}$, where $B_{\rho}$ is the ball of radius $\rho$, centered at the origin

\NewAdd{	Note that the classes of \NewAdd{convic} functions and discrete \NewAdd{convic} functions are quite wide and natural. Here we give a few arguments in support of this thesis \cite{JOGO}.
	
As already mentioned, the class of (discrete) \NewAdd{convic} functions contains the class of (discrete) convex functions and the class of (discrete) strictly quasi-convex functions. On the other hand, there exist \NewAdd{convic} functions that are not convex nor strictly quasiconvex, e.g., $\max\{1,\, \sqrt{x}\}$.
Note that the maximum of \NewAdd{convic} functions is \NewAdd{convic}. 
If $f$ is a non-decreasing \NewAdd{convic} function and $g$ is \NewAdd{convic} then $f(g(x))$ is \NewAdd{convic}. 
Any affine transformation of coordinates transforms a \NewAdd{convic} function into a \NewAdd{convic} one.

The problem of minimization of $f(x)$, s.t. $g_i(x) \leq 0$ ($i = 1,2,\dots,m$),  $x \in \ZZ^n$, where $f$, $g_i$ are \NewAdd{convic} (or discrete \NewAdd{convic}), can be reduced to the problem of integer minimization of the \NewAdd{convic} (or, respectively, discrete \NewAdd{convic}) function 
$$
\max\{0,\, f(x), \, M\cdot g_1(x), \dots, M\cdot g_m(x)\},
$$
where the positive constant $M$ is sufficiently large and in many cases can be simply computed from the input of the problem. 
In particular, the integer linear programming problem can be reduced to the problem of minimizing a discrete \NewAdd{convic} function. Note that using lexicography it is possible to do this reduction without using constant $M$ \cite{JOGO}.

Note that not every discrete \NewAdd{convic} function can be extended to a \NewAdd{convic} one \cite{JOGO}. On the other hand, if it is even extended, we may not be able to find the extension, if the function is given by the oracle or in other cases.
For example, the function $f(x)=2\rho^3 x_2-x_1^3$ in the ball $B_{\rho}$, $\rho>1$, is not convex, but the induced function defined in integer points of the ball is discrete \NewAdd{convic}. Hence, the algorithm in \cite{JOGO} is not applied here and a new algorithm for minimizing discrete \NewAdd{convic} functions is required.
}

For convenience of our further exposition, we reduce the problem under consideration to the problem of finding the lattice vector which is minimum with respect to the linear order given by a discrete \NewAdd{convic} function.

Note that any function $f: ~ D \to \RR$ induces some linear order on the set $D$:
$$
x \preceq y \qquad \Leftrightarrow \qquad f(x) \le f(y).
$$
Moreover, instead of considering the functions $f$, we can investigate special linear orders defined on the set $D$.

Let $L$ be a lattice in $\ZZ^n$.
Consider a linear order $\preceq$ on the set $D = L \cap B_{\rho}$ with the following property:
for every $z \in y + \cone \{y-x: ~ x \preceq y \}$, it holds that $y \preceq z$.
There is an oracle capable for any $x, y \in D$ to verify the truth of the statement $x \preceq y$.
It is required to find a minimum point with respect to this order.

\section{The $\beta$-rounding algorithm}\label{sec_beta_rounding}

In this section, we give a general description of Lenstra's method (see \cite{Grotschel,Schrijver}).
Its main idea is constructing {\em $\beta$-rounding} of the set $M$, i.e. constructing the ellipsoid 
$$
E(A,a)=\{x\in \QQ^n:~ (x-a)^{\top}A^{-1}(x-a)\le 1\},
$$ 
such that 
\begin{equation}
E(\beta^2A,a)\subseteq M\subseteq E(A,a).\label{l0} 
\end{equation}

If $E(A,a)$ is sufficiently ``large'', then $E(\beta^2A,a)$ contains an integer point belonging to $M$. 

If  $E(A,a)$ is ``small'', then one can specify a number $k(\beta)$ that does not depend on $M$, 
such that all integer points in $M$ are located on at most $k(\beta)$ hyperplanes, 
therefore, the $n$-dimensional problem reduces to $k(\beta)$ problems in dimension $n-1$, and the recursion can be applied. 

Given a sequence of vectors 
\begin{equation}
c_0,~c_1,~\dots \label{c}
\end{equation}  
and numbers $R$, $\epsilon$, 
the algorithm for $\beta$-rounding of the set $M$ % \cite{ Grotschel} 
finds the sequence 
$$
(A_0,a_0)=(R^2I,0),~(A_1,a_1),~\dots
$$ 
such that $E(A_{k+1},a_{k+1})$ is the ellipsoid with minimum volume containing the set 
$\{x\in E(A_k,a_k):~c^{\top}_kx \le c^{\top}_ka_k+\beta\sqrt{c^{\top}A_kc} \}.$

The algorithm finishes its work at the $k$-th step if at least one of the following conditions is fulfilled:
\begin{enumerate}
	\item[1)] the sequence (\ref{c}) is terminated;

	\item[2)] the pair $(A_k,a_k)$ satisfies (\ref{l0});

	\item[3)] $\det A_k\le \epsilon$. 
\end{enumerate}

For $\beta=1/(n+1)$, the number of iterations of the algorithm is at most 
\begin{equation}
N_{\epsilon}=5n(n+1)^2|\log(\epsilon)| + 5n^2 (n + 1)^2|\log(2R)| + \log(n + 1)+1 \label{N}
\end{equation}
(see Theorem~3.3.9 in \cite{Grotschel}).

\section{The cutting hyperplane construction}\label{sec_hyperplain}

To use the $\beta$-rounding procedure one must be able to find the cutting plane for any ``sufficiently large'' ellipsoid
$$
E(A,a)=\{x\in \QQ^n:~ (x-a)^{\top}A^{-1}(x-a)\le 1\}.
$$
We show how to do this.

Here we use the ideas of shallow cuts introduced by Nemirovsky and Yudin \cite{NemirovskyYudin1983}.

We will use the dot product $(x,y)=x^{\top}A^{-1}y$ and the induced norm $\|x\|=\sqrt{(x,x)}$. 
In such a metric, the ellipsoid $x^{\top}A^{-1}x\le r^2$ is a ball $B_r$.

\begin{lemma}\label{lemma1}%
Suppose that $r<R<1$, a point $x'$ belongs to the boundary of $B_R$, and
$K'=x'+\cone\{x'-x:~ x\in B_r\}$.
Then for any point $x\in B_1\setminus K'$ it holds that
$$
(x',x)\le r^2+\sqrt{(1-r^2)(R^2-r^2)}.
$$
\end{lemma}

\begin{proof}
Consider the set $S$ of all points $y\in B_r$, for which the ray $\ell = \{x'+ t\cdot(x'-y):~ t \ge 0\}$ is boundary for the cone $K'$. 
Since the extensions of all such rays touch the ball $B_r$, $(y,x'-y)=0$ for all $y\in S$.

Let us verify that the ray $\ell$ intersects the boundary of the ball $B_1$ at a point $z = x'+ t\cdot(x'-y)$ with 
$$
t=\sqrt{\frac{1-r^2}{R^2-r^2}}-1. 
$$
Indeed, for $y\in S$ we have
\begin{multline*}
\|z\|^2 = \|x'+t\, (x'-y)\|^2=\|(t+1)\, (x'-y)+y\|^2= \\
    = (t+1)^2\|y-x'\|^2 + 2(t+1)\, (x'-y,y) + \|y\|^2 = \\
    = \frac{1-r^2}{R^2-r^2} (\|x'\|^2-\|y\|^2) +r^2 =\frac{1-r^2}{R^2-r^2} (R^2-r^2) +r^2=1.
\end{multline*}

Since $z$ lies on the boundary ray of the cone $K'$, for any point $x\in B_1 \setminus K'$ we have
\begin{multline*}
(x',x) \le (x',z) = \bigl(x',\, x'+t\, (x'-y)\bigr)= \|x'\| + \bigl(x',\, t\,(x'-y)\bigr)= \\
    =R^2+t\,(x'-y,\, x'-y)=r^2+\sqrt{(1-r^2)(R^2-r^2)}.
\end{multline*}
\qed
\end{proof}

\NewAdd{	
\begin{lemma}\label{lemma2}%
Let $g_1,\dots,g_n$ be a basis of a lattice $L$, 
$\theta=\max\limits_{i=1,\dots,n}\|g_i\|$ and $R\ge r+\theta n$, 
then $B_r\subset \conv(B_R\cap L)$.
\end{lemma}
}
\begin{proof}
\NewAdd{
		For $x=x_1g_1+\dots+x_ng_n$ we denote 
		$$
		\lfloor x\rceil =
		\{y_1g_1+\dots +y_ng_n:~ \lfloor x_j\rfloor \le y_j\le \lceil x_j\rceil ~ (j=1,\dots,n)\}.
		$$
}
If $x\in B_r$, then
$$ 
\lfloor x\rceil\subseteq B_r+ \{\alpha_1 g_1 + \dots + \alpha_n g_n: |\alpha_j|<1 ~ (j=1,\dots,n) \}\subseteq B_{r+\theta n}\subseteq B_R. 
$$
Since $x\in \lfloor x\rceil$, $x\in \conv(B_R\cap L)$.
\qed
\end{proof}

\NewAdd{Denote
		$$
		\beta=\frac1{n+1},\quad
		\sigma=\frac{2\beta^3}{27n},\quad
		R=\frac{\beta}3,\quad
		r=\frac{R\sqrt{1-\beta^2}}{\sqrt{1-2\beta R+R^2}}.
		$$

\begin{lemma}\label{lemma3}%
$R\ge r+\sigma n.$
\end{lemma}
}
\begin{proof}
\begin{multline*}
  R - r = \frac{\beta}{3} - \frac{\beta\sqrt{1-\beta^2}}{3\sqrt{1-2\beta^2/3 + \beta^2/9}} = \\
	      = \frac{\beta}{3} - \frac{\beta\sqrt{1-\beta^2}}{\sqrt{9-5\beta^2}}
				= \beta\cdot\frac{\sqrt{9-5\beta^2} - 3\sqrt{1-\beta^2}}{3\sqrt{9-5\beta^2}} = \\
				= \beta\cdot\frac{4\beta^2}{3\sqrt{9-5\beta^2}(\sqrt{9-5\beta^2}+3\sqrt{1-\beta^2})}
				\ge \frac{2\beta^3}{27} \ge \sigma n.
\end{multline*}\qed
\end{proof}

%\begin{eqnarray*}
%R-r=R\frac{\sqrt{1-2\beta R+R^2}-\sqrt{1-\beta^2}}{\sqrt{1-2\beta R+R^2}}= \\ \frac{R(\beta-R)^2}{(\sqrt{1-2\beta R+R^2}+\sqrt{1-\beta^2})\sqrt{1-2\beta R+R^2}}\ge \\
%\frac{R(\beta-R)^2}{2(1-2\beta R+R^2)}\ge\frac{1}{2} R(\beta-R)^2=\frac{2\beta^3}{27} 
%\ge n\tau \end{eqnarray*}

%\NewAdd{Объединил Леммы 4 и 5.}	

%\begin{lemma}\label{lemma4}%
%	Let $g_1,\dots,g_n$ be a basis of a lattice $L\subseteq \ZZ^n$,
% $\theta=\max\limits_{i=1,\dots,n}\|g_i\|$,
%	$\theta\le \sigma$. Then $B_r\subset conv(B_R\cap L).$
%\end{lemma}
%
%\begin{proof}
%Follows from lemmas~\ref{lemma2} and \ref{lemma3}.	\qed
%\end{proof}

\NewAdd{The following lemma describes how to construct a shallow cut.}

\NewAdd{
\begin{lemma}\label{lemma5}%
	Let $g_1,\dots,g_n$ be a basis of a lattice $L\subseteq \ZZ^n$, $\theta=\max\limits_{i=1,\dots,n}\|g_i\|$,
	$\theta\le \sigma$.
	Then $B_r\subset \conv(B_R\cap L)$.
	Further, let $x^*$ be the maximum point in $B_R\cap L$ with respect to $\preceq$, 
	a point $x^0$ be the minimum point in $B_1\cap \ZZ^n$ with respect to $\preceq$, and $E$ be the minimum volume ellipsoid containing the set 
	$\{x\in B_1: (x^*,x)\le \beta \|x^*\|\}$.
	Then $x^0\in E$.
\end{lemma}
}

\begin{proof}
\NewAdd{The inclusion $B_r\subset \conv(B_R\cap L)$ follows from lemmas~\ref{lemma2} and \ref{lemma3}.}

Let $x'$ be the intersection point of the ray $\{tx^*: ~ t \ge 0 \}$ and the boundary of the ball $B_R$.

Consider the following three cones (see Fig.\,\ref{cutting_lemma_fig}):
\begin{gather*}
K_0 = x^* + \cone\{x^*-x:~ \NewAdd{x\in B_R \cap L\},} \\
K^* = x^* + \cone\{x^*-x:~ x\in B_r \}, \quad
K'  = x'  + \cone\{x'-x:~ x\in B_r \}.
\end{gather*}
It is obvious that $K' \subseteq K^*$
\NewAdd{
and
%Using Lemma~\ref{lemma4}, we obtain
 $B_r \subset  \conv(B_R\cap L)  \subseteq \conv(B_R\cap\ZZ^n)$,
hence $K^* \subseteq K_0$.
}

Since $x^*$ is the maximum point in $B_R\cap L$ with respect to $\preceq$, 
%using properties 1,\,2 of this order, 
we derive $x^* \preceq x$ for all $x\in K_0$. 
But the cone $K_0$ contains $K'$, hence the minimum point in $B_1\cap L$ with respect to $\preceq$ 
belongs to $B_1\setminus K'$.
By Lemma~\ref{lemma1}, for any point $x\in B_1\setminus K'$ we have
$(x',x)\le r^2+\sqrt{(1-r^2)(R^2-r^2)}$,
hence
\begin{multline*}
    \frac{(x^*,x)}{\|x^*\|} = \frac{(x',x)}{\|x'\|} = \frac{(x',x)}{R} \le \frac{r^2+\sqrt{(1-r^2)(R^2-r^2)}}{R} = \\
		  = \frac{R(1-\beta^2)}{1 - 2\beta R + R^2} + \frac{\sqrt{(1-2\beta R + R^2\beta^2)(R^4 + R^2\beta^2-2\beta R^3)}}{R(1-2\beta R + R^2)} = \\
		  = \frac{R(1-\beta^2)}{1 - 2\beta R + R^2} + \frac{(1-\beta R)(\beta - R)}{1 - 2\beta R + R^2} 
		  = \frac{\beta(1 - 2\beta R + R^2)}{1 - 2\beta R + R^2} = \beta.
\end{multline*}
\qed
\end{proof}

\begin{figure}[h]\label{cutting_lemma_fig}
	\centering
	\includegraphics[scale=0.31]{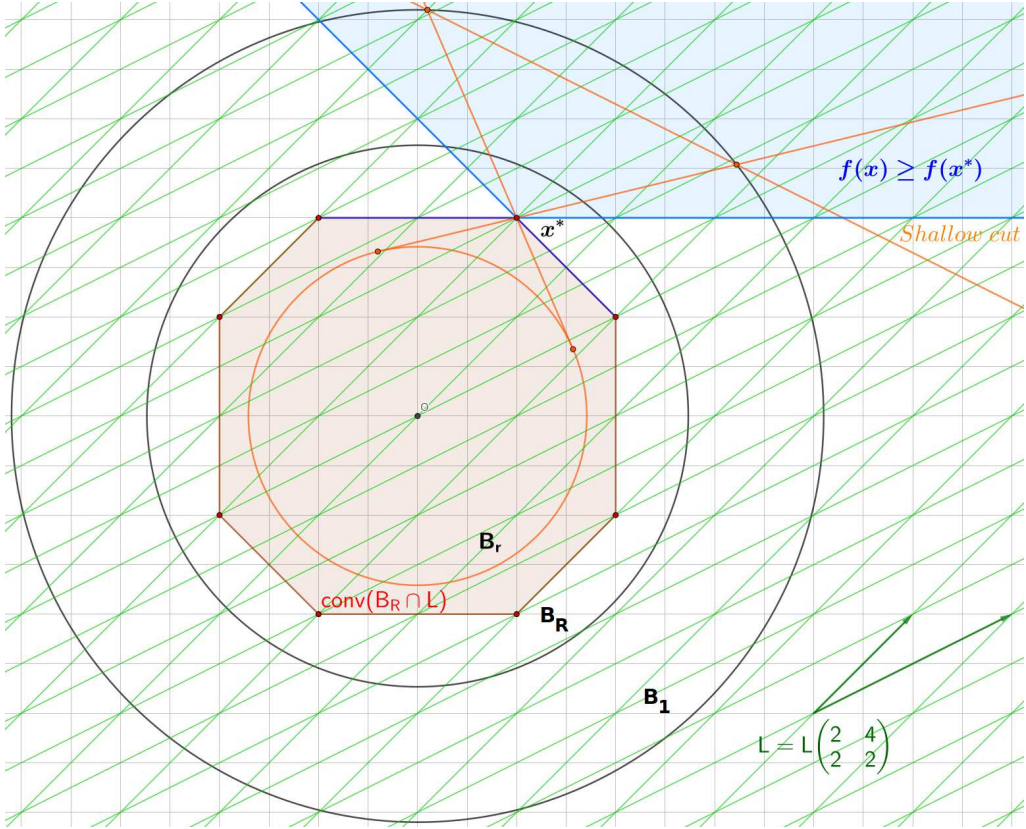}
	\caption{An illustration for Lemma \ref{lemma5}. Here, we have $f : \ZZ^2 \to \RR$ and the lattice $L$ with the basis $\dbinom{2}{2}$ and $\dbinom{4}{2}$. For integral points from the blue cone we have $f(x) \geq f(x^*)$, the shallow cut is denoted by the orange line.}
\end{figure}

%\GribAdd{
%Определение ортогонального дефекта пришлось перенсти сюда, так как первый раз оно используется именно здесь. (Так было и раньше, до введения леммы 6)
%}

Let \NewAdd{$G = (g_1, g_2, \dots, g_n)$ be a basis of a lattice $L\subseteq\ZZ^n$. 
Denote by 
$$
\defect G =\frac{1}{\det L} \prod_{i=1}^n \| g_i\|_{A^{-1}}
$$  
the orthogonality defect of $G$.}

%\GribAdd{
%Добавил Лемму \ref{lemma6}. Она нужна для того, чтобы  отметить сложность потсроения shallow cut. Также она закрывает комментарий 5 второго рецензента, который просил оформить построение  shallow cut, как отдельную лемму. Раньше эту роль играла Лемма \ref{lemma5}, но в ней не говорится ничего о сложности.
%}

%We can use Lemma~\ref{lemma5} in the procedure $\beta$ - rounding, however, for this it is necessary to define the lattice $ L $ so as "efficiently" to find $x^*.$  Now we show how to do it.

\NewAdd{The following lemma helps us to estimate the time need to construct the shallow cut.}

\NewAdd{
\begin{lemma}\label{lemma6}
Let $G = (g_1, g_2, \dots, g_n)$ be a basis of $\ZZ^n$, $\theta=\max\limits_{i=1,\dots,n}\|g_i\|$ and $\theta \le \sigma$. 
Let $x^0$ be the minimum point in $B_1 \cap \ZZ^n$ with respect to $\preceq$. Then, the shallow cut $(x^*,x) \leq \beta \|x^*\|$ with the property 
$$
x^0 \in \{x \in B_1 : (x^*, x) \leq \beta \|x^*\|\}
$$ 
can be computed by an algorithm with the oracle complexity $2^{O(n \log n)} \defect G$ and the bit complexity $2^{O(n \log n)} \defect(G) \poly(s)$, where $s = \size G + \size A$ is the input size.
\end{lemma}
}
\begin{proof}
%\GribAdd{
%В следующем обзаце перевел на английский некторые слова Сергея Ивановича и заменил базис буквы базиса $h_1,h_2, \dots, H$ на буквы $g_1,g_2, \dots, G$ так как именно такие буквы используются в лемме \ref{lemma5}.
%}
Let $G^*=(g^*_1,g^*_2,\dots,g^*_n)$ be the Gram--Schmidt basis corresponding to the lattice basis $G$. Let $T$ be the upper triangular matrix such that $G=G^*T$ and let 
$$
d_i=\left\lfloor \frac{\sigma}{\|g_i\|} \right\rfloor \qquad (i=1,\dots,n).
$$
It is clear that we have $\|d_i g_i\| \leq \sigma$.

Consider the lattice generated by the vectors $d_i h_i$ $(i = 1, \dots, n)$. 
To find the maximum point $x^*$ in $B_R\cap L$ with respect to $\preceq$ it is enough to enumerate all integer solutions to the inequality $x^{\top}A^{-1} x \le R^2$. Or, after replacement $x=G \diag(d_1,\dots,d_n)y=G^*T \diag(d_1,\dots,d_n)y$, it is enough to enumerate all integer solutions to the inequality 
$$
\sum_{i=1}^{n} d^2_i||g^*_i||^2(y_i+t_{i,i+1}y_{i+1}+\dots+t_{i,n}y_n)^2\le R^2.
$$

%\GribAdd{
%Следующую формулу чуть чуть изменил, так как там раньше использовались буквы $h$ вместо $g$. Также оставил дефект $\defect G$ в формуле, так как: 1) так лемма выглядит более общей 2) в описании алгоритма мы используем дефект К-З базиса еще раз, чтобы не упомянать одну и ту же ссылку дважды, решил оставить дефект здесь.
%}

They are at most
\begin{multline} 
 	\prod_{i=1}^n \left(\frac{2R}{d_i\|g^*_i\|}+1\right)\le \prod_{i=1}^n \left(\frac{4R}{d_i\|g^*_i\|}\right)=\\
 	=\prod_{i=1}^n \left(\frac{4R}{d_i\|g_i\|}\right)\prod_{i=1}^n \left(\frac{\|g_i\|}{\|g^*_i\|}\right)
 	=\prod_{i=1}^n \left(\frac{4R}{d_i\|g_i\|}\right) \defect G\le \\ \le \left(\frac{8R}{\sigma}\right)^n \defect G\le %\left(\frac{4R}{d_i\|h_i\|}\right)^nn^n\le
 	\bigl(36n(n+1)^2+1\bigr)^n n^n
 	=2^{O(n\log n)} \defect G
 	\label{l6} 
\end{multline}  
integer points. 
This procedure takes $2^{O(n\log n)} \defect G$ calls to the oracle. 
 In deriving the inequality in (\ref{l6}) we used that
 $$
 d_i \cdot \|g_i\| = \| g_i\| \cdot \left\lfloor\frac{\sigma}{\| g_i\|}\right\rfloor \ge \frac{\sigma}{2}.
 $$
 The last inequality here follows from the inequalities
 $$
 \left\lfloor \frac{u}{v}\right\rfloor \ge \left\{\frac{u}{v}\right\} 
 \quad\Rightarrow\quad
 2\left\lfloor\frac{u}{v}\right\rfloor\ge \frac{u}{v}
 \quad\Rightarrow\quad
 v\left\lfloor\frac{u}{v}\right\rfloor \ge \frac{u}{2},
 $$ 
 which are valid for $u\ge v>0$. 
 % \GribAdd{
 %Since all intermediate operations have polynomial complexity, the total bit %complexity is correct.
 %}
\end{proof}
	
\NewAdd{
\begin{note} For the sake of simplicity, we considered only ellipsoids with centers in $0$. However, the enumeration method, described in Lemma \ref{lemma6}, can be applied for ellipsoids with arbitrary centers $(x-a)^\top A^{-1} (x-a) \le R^2$.  
\end{note} 
}

\section{The algorithm}\label{sec_algorithm}
 
\NewAdd{
 Before we give the main minimization algorithm we describe a preprocessing procedure that will be very helpful, when we need to reduce the dimension of an initial problem and to find a short lattice basis in the reduced space.} 
\NewAdd{Here we fully follow Dadush' IP-Preprocessing Algorithm \cite[pp.\,223--225]{Dadush}.}

% \GribAdd{
% Следующая лемма есть в работе Дадуша \cite{Dadush}, но немного в другой формулировке. В принципе ее можно даже и не доказывать, а просто написать правильные слова, о том, что она верна и сослаться на Дадуша. Его оригинальное доказательство, переделанное на наш манер я привел пурпурным цветом после формулировки леммы.
% 
% Лемма эта нужна, чтобы строить короткий базис при переходе к следующей размерности и поддерживать инвариант о том, что мы находимся все время в шарике радиуса $R$. Все это необходимо, чтобы правильно оценить битовую сложность.}

\NewAdd{ 
%The following lemma almost literally repeats the statement in \cite[pp.\,223--225]{Dadush}.

\begin{lemma}\label{algorithm1} Let $L$ be an $n$-dimensional lattice given by a basis $B \in \QQ^{n \times n}$, and $H = \{x \in \RR^n : A x = b\}$ be an affine subspace, where $A \in \QQ^{m \times n}$ and $b \in \QQ^m$. 
Let also $E = a_0 + B_{\rho}$, where $a_0 \in \QQ^n$ and $\rho \in \QQ_+$. Then, there is an algorithm $\PrepAlgName$ with the bit complexity $2^{O(n \log n)} \poly(s)$, where $s$ is input size, which either decides that $E \cap L \cap H = \emptyset$ or returns
\begin{enumerate}
\item[\rm 1)] a shift $p \in L$,
\item[\rm 2)] a sublattice $L' \subseteq L$, $\dim L' = k \leq n$, given by a basis $b_1',\dots,b_k'$,
\item[\rm 3)] a vector $a_0' \in \linh L'$ and a radius $\rho'$, 
      $0 < \rho' \leq \rho$,
\end{enumerate} 
satisfying the following properties
\begin{enumerate}
\item[\rm 1)] $E \cap L \cap H = (E' \cap L') + p$, where 
$E' = \{x \in \linh L' :~ \|x - a_0'\| \leq \rho' \}$,
\item[\rm 2)] $\max\limits_{1 \leq i \leq k} \|b_i'\| \leq 2 \sqrt{k} R'$,
\item[\rm 3)] $a_0'$, $\rho'$, $b_1', \dots, b_k'$ and $p$ have polynomial in $s$ encodings.
\end{enumerate}
\end{lemma}
}
\begin{proof}
\NewAdd{See~\cite[pp.\,223--225]{Dadush}.}
\qed
\end{proof}

\NewAdd{The algorithm for minimizing a discrete \NewAdd{convic} function is given 
on Fig.\,\ref{algorithm2}.
On the step 6 we construct the reduced Korkin--Zolotarev basis following \cite{Schnorr}.}

\begin{figure}\label{algorithm2}
	\NewAdd{
\algsetup{indent=2em}
\begin{algorithmic}[1]
\REQUIRE The comparison oracle for the order $\preceq$; a lattice $L \subseteq \ZZ^n$; a point $a_0 \in \QQ^n$ and a radius $\rho \in \QQ_+$; a rational affine subspace $H$. 
\ENSURE Return EMPTY if the set $(a_0 + B_{\rho}) \cap L \cap H$ is empty. If it is not, return the minimum point $x^*$ with respect to $\preceq$ in the set 
$(a_0 + B_{\rho}) \cap L \cap H$.
\STATE $(a_0,\rho,L,p) : = \PrepAlgName(a_0,\rho,L,H)$. 
\STATE Set $p$ as the origin, when we call the $\preceq$ comparison oracle.
% Or in other words set $COMP_\preceq(x) := COMP_\preceq(x-p)$.
\STATE $E = \{x \in \linh L :~ \|x - a_0\|_2 \leq \rho\}$, $n := \dim L$, $\beta := \frac{1}{n+1}$, $\sigma:=\frac{2 \beta^3}{27n}$.
\REPEAT
\STATE Suppose that $E = \{x \in \linh L :~ (x - a)^\top A^{-1} (x - a) \leq 1 \}$.
\STATE Find the reduced Korkin--Zolotarev basis $h_i$ $(i=1,\dots,n)$ of the lattice $L$ with respect to the norm $\|x\|_{A^{-1}} = \sqrt{x^\top A^{-1} x}$, defined by the ellipsoid $E$. 
\STATE Find $\tau=\max\limits_{i=1,\dots,n}\|h_i\|_{A^{-1}}$.
\IF{$\tau \leq \sigma$}
\STATE With help of Lemma~\ref{lemma6} construct cutting plain and perform one iteration of $\beta$-rounding algorithm from \cite{Grotschel} for the ellipsoid $E$ with shallow cut $(x^*,x)_{A^{-1}} \leq \beta \|x^*\|_{A^{-1}}$. 
Let $E$ be the resulting ellipsoid. %If $|\det A_{k+1}|< \sigma^{-n}$ (and hence $\tau>\sigma$) then goto step 3, else go to step 1.
\ENDIF
\UNTIL{$\tau > \sigma$}
\STATE Let $H=(h_1,..,h_{n-1})$. Find irreducible integer vector $d \in \linh L$, such that $d^{\top} H=0$.
\FORALL{$\alpha \in \{ d^\top x :~ x \in E\} \cap \ZZ$} \label{enumeration_step}
\STATE $H_\alpha := \{x \in \linh L:~ d^\top x = \alpha\}$.
\STATE Call $\MainAlgName(a_0,\rho,L,H_\alpha)$.
\ENDFOR
\RETURN If all recursive calls of the $\MainAlgName$ algorithm have returned EMPTY, then return EMPTY. In the opposite case, return $p + y$, where $y$ is a minimum point with respect to $\preceq$ between all recursive calls of the $\MainAlgName$ algorithm.
\end{algorithmic}
}
\caption{The algorithm $\MainAlgName$ for minimizing a discrete \NewAdd{convic} function}
\end{figure}

%\GribAdd{
% Чуть чуть изменил формулировку основное теоремы, изменил $\rho$ на $R$ так как это общеприянятое обозначение. \NewAdd{Снова заменил R на $\rho$. Причины этого уже обсуждались.} Также добавил битовую сложность.}
 
%Let us prove the main theorem.

\NewAdd{
\begin{theorem}
There exists an algorithm for minimizing a discrete \NewAdd{convic} function, given by the comparison oracle on $B_{\rho} \cap \ZZ^n$, using $2^{O(n^2 \log n)} \log \rho$ calls to the oracle. The bit complexity of the algorithm is $2^{O(n^2 \log n)} \poly(\log \rho)$.
\end{theorem}
}
\begin{proof}
\NewAdd{
Let $\preceq$ be an order on $\ZZ^n$ induced by the function $f$. To find a minimum point we need to call algorithm $\MainAlgName$ with input parameters $a_0 = 0$, $\rho$, $L = \ZZ^n$ and $H = \{x \in \RR^n : 0^\top x = 0\}$. 

The correctness of the algorithm follows directly from the fact that we enumerate all possible hyperplanes $H_\alpha := \{x \in \linh L:~ d^\top x = \alpha\}$ in line \ref{enumeration_step}, where $\alpha \in \{ d^\top x : x \in E\} \cap \ZZ$, while the ellipsoid $E$ localizes some minimum point.
}

%\GribAdd{
% Следующую фразу Сергея ивановича предлагаю заменить на оранжевый текст идущий после нее. Сделать я предлагаю по следующим причинам: 1) не понятно, откуда берется неравенство на детерминант (я добавил пояснение) 2) не понятно, почему $\epsilon = \sigma^{-n}$. Нужно добавить пояснение, что мы сами берем такое эпсилон, чтобы после полиномиального числа отсечений выполнилось нераенство $\tau > \sigma$. Сейчас это понятно, но разобраться было не просто.
% }
% 
% Let us estimate the number of the calls to the oracle sufficient for reduction to a smaller dimension. 
%Obviously,  $|\det\, A_{k+1}|\ge \tau^{-n}$. 
%Since $\epsilon=\sigma^{-n}$, we obtain from (\ref{N}) the bound $O(n^4\log n\rho)$ for the number of iterations in the $\beta$-rounding procedure.
%
%Multiplying this value by the right-hand side of (\ref{l6}), we obtain that the number of calls to the oracle is 
%$2^{O(n\log\,n)}\log\,n\rho=2^{O(n\log\,n)}\log\,\rho$.

\NewAdd{
Let us prove the claimed oracle complexity estimate. Firstly, let us estimate the number of iterations of the repeat-until loop in lines 4--10. Let $H = (h_1, h_2, \dots, h_n)$, then
$$
\det (A^{-1}) = \det (H^\top A^{-1} H) \leq \|h_1\|_{A^{-1}} \dots \|h_n\|_{A^{-1}} \leq \tau^n.
$$ 

Hence, when $\det A$ becomes less then $\sigma^{-n}$, we will have $\tau > \sigma$, and the loop will be finished. Taking $\epsilon=\sigma^{-n}$, we obtain the bound $O(n^4 \log (n \rho))$ for the number of iterations in the $\beta$-rounding procedure.

By the Lemma \ref{lemma6}, the oracle complexity of the step 9 is $2^{O(n \log n)} \defect H$. 
}

It is known \cite{Lagarias1990}, that the upper bound for the orthogonality defect of a Korkin--Zolotarev reduced basis is $2^{O(n\log n)}$. \NewAdd{Hence, the oracle complexity of the repeat-until loop in lines 4--10 become $2^{O(n \log n)}$. }

Suppose that $\tau > \sigma$. As in \cite[p. 258, formula (56)]{Schrijver}, 
it can be shown\footnote{\NewAdd{In \cite[p. 258, formula (56)]{Schrijver}, LLL basis is discussed, but the proof of the inequality is valid for any basis, including not reduced one.}} that  $|d^{\top}x|\le \defect H/\tau$ for any $x\in E(A,a)$. Hence, the initial problem is reduced to solving at most 
$$
2 \defect H/\sigma+1= 27 n(n+1)^3 \defect H =2^{O(n\log n)} 
$$ 
problems in the dimension $n-1$.

Denoting by $\varphi(n,\rho)$ the total number of calls to the oracle, we obtain
$$
\varphi(n,\rho) = 2^{O(n \log n)} \log \rho + 2^{O(n\log n)} \varphi(n-1, \rho) =  2^{O(n^2 \log n)} \log \rho.
$$
\NewAdd{Now, to derive the claimed bound for the bit complexity, we could repeat the arguments from Dadush' thesis \cite[p.\,231]{Dadush}.}
%\GribAdd{
%Дальше пурпурным цветом идет объяснение заявленной оценки битовой сложности, слизанное дословно у Дадуша \cite[p.\,231]{Dadush}. Надо будет тут переписать.
%\NewAdd{Просто выкидываем.}
%}
%
%
%
%\DadushAdd{ Let us estimate the total bit-complexity. We note that dimension decreases at each node of recursion, and the algorithms executed during node processing (Preprocessing, the K-Z basis construction, the shallow cut construction by Lemma \ref{lemma6}) have either polynomial or $2^{O(n \log n)}$ dependence on dimension. Therefore, the only thing we need to check is that the encoding length of the subspaces and bases used at each recursion node (which the algorithm have polynomial dependence on) have encoding sizes bounded by a fixed polynomial in the original input. We note that after each call to Preprocessing, the basis returned have $l_2$ length at most $2 \sqrt{n} \rho$. Now we note that any lattice vector in $L \subseteq \ZZ^n$ of $l_2$ norm $2 \sqrt{n} \rho$ has encoding length bounded by $\poly(\size \rho, n)$. Hence, immediately after each call to Preprocessing algorithm, the current basis indeed has encoding length bounded by a fixed polynomial in the input. Finally, it is easy to see that in between calls to the Preprocessing algorithm, the encoding size of all the computed parameters grows by at most a fixed polynomial as well. Hence, all computations during the execution of algorithm occur on inputs having polynomial encoding size as needed. }
%
\qed
\end{proof}

%Now we give a general description of the algorithm for finding the point minimal with respect to a given order on the set 
%$\{x\in \ZZ^n:x^\top x\le \rho^2 \} $.

%Let us estimate the number of the calls to the oracle sufficient for reduction to a smaller dimension. 
%Obviously,  $|\det\, A_{k+1}|\ge \tau^{-n}$. 
%Since $\epsilon=\sigma^{-n}$, we obtain from (\ref{N}) the bound $O(n^4\log n\rho)$ for the number of iterations in the $\beta$-rounding procedure.
%
%Multiplying this value by the right-hand side of (\ref{l6}), we obtain that the number of calls to the oracle is 
%$2^{O(n\log\,n)}\log\,n\rho=2^{O(n\log\,n)}\log\,\rho$.
%
%
%{\color{red}пїЅпїЅпїЅпїЅпїЅпїЅпїЅ пїЅпїЅпїЅпїЅпїЅпїЅ пїЅпїЅпїЅпїЅпїЅпїЅпїЅпїЅ пїЅпїЅпїЅ пїЅпїЅпїЅпїЅпїЅпїЅпїЅпїЅ пїЅпїЅпїЅпїЅпїЅпїЅпїЅпїЅпїЅ, пїЅ пїЅпїЅпїЅпїЅпїЅпїЅпїЅ пїЅпїЅ пїЅпїЅпїЅпїЅ $n$ - пїЅпїЅпїЅпїЅпїЅпїЅ пїЅпїЅпїЅпїЅпїЅпїЅ
%пїЅпїЅпїЅпїЅпїЅпїЅпїЅпїЅ пїЅ $(n-1)$  - пїЅпїЅпїЅпїЅпїЅпїЅ  пїЅпїЅ пїЅпїЅпїЅпїЅпїЅпїЅпїЅпїЅпїЅ $\{x\in \ZZ^n:x^\top x\le \rho^2 , d^Tx=\alpha\}$.}
%
% Denoting by $\varphi(n,\rho)$ the total number of calls to the oracle, we obtain
%$$
%\varphi(n,\rho) = 2^{O(n \log n)} \log\rho + 2^{O(n\log n)} \varphi(n-1, \rho) =  2^{O(n^2 \log n)} \log \rho.
%$$         
% 

\NewAdd{
\paragraph{Acknowledgment}
The work was supported by the Russian Science Foundation Grant No. 17-11-01336.

The authors are extremely grateful to the reviewers for their careful reading of the manuscript and suggesting making changes, which greatly improved the paper.
}

\end{document}